\documentclass[11pt]{article}

\usepackage{amsmath}
\usepackage{amsthm}
\usepackage{amssymb}
\usepackage{latexsym}
\usepackage{mathabx}
\usepackage{pstricks}
\usepackage{graphicx, pst-plot, pst-node, pst-text, pst-tree}
\usepackage{titlefoot}
\usepackage{enumerate}
\usepackage{titlesec}
\usepackage{units}
\usepackage[small,it]{caption}

\usepackage{color}
\definecolor{lightgray}{rgb}{0.9, 0.9, 0.9}
\definecolor{darkgray}{rgb}{0.7, 0.7, 0.7}
\definecolor{darkblue}{rgb}{0, 0, .4}

\usepackage[bookmarks]{hyperref}
\hypersetup{
	colorlinks=true,
	linkcolor=darkblue,
	anchorcolor=darkblue,
	citecolor=darkblue,
	urlcolor=darkblue,
	pdfpagemode=UseThumbs,
	pdftitle={Finding Regular Insertion Encodings for Permutation Classes},
	pdfsubject={Combinatorics},
	pdfauthor={Vincent Vatter},
	pdfkeywords={permutation class, rational generating function, regular language}
}

\newtheorem{theorem}{Theorem}
\newtheorem{proposition}[theorem]{Proposition}

\newcounter{todocounter}

\makeatletter
\newfont{\footsc}{cmcsc10 at 8truept}
\newfont{\footbf}{cmbx10 at 8truept}
\newfont{\footrm}{cmr10 at 10truept}
\renewenvironment{abstract}%
		{
		  \begin{list}{}%
		     {\setlength{\rightmargin}{1in}%
		      \setlength{\leftmargin}{1in}}%
		   \item[]\ignorespaces\begin{small}}%
		 {\end{small}\unskip\end{list}}

\datefoot{\today}
\amssubj{05A05, 05A15, 68Q45}

\newpagestyle{main}[\small]{
	\headrule
	\sethead[\usepage][][]
	{\sc Finding Regular Insertion Encodings for Permutation Classes}{}{\usepage}}

\title{\sc{Finding Regular Insertion Encodings for Permutation Classes}}
\author{\sc{Vincent Vatter}\\
\small Department of Mathematics\\[-1pt]
\small Dartmouth College\\[-1pt]
\small Hanover, New Hampshire USA\\[-10pt]}

\titleformat{\section}
	{\large\sc}
	{\thesection.}{1em}{}	

\titleformat{\subsection}
	{\normalsize\sc}
	{\thesubsection.}{1em}{}	

\date{}

\begin{document}
\maketitle

\pagestyle{main}

\newcommand{\Av}{\operatorname{Av}}
\newcommand{\Age}{\operatorname{Age}}
\newcommand{\A}{\mathcal{A}}
\newcommand{\C}{\mathcal{C}}
\newcommand{\D}{\mathcal{D}}
\newcommand{\E}{\mathcal{E}}
\newcommand{\HH}{\mathcal{H}}
\newcommand{\I}{\mathcal{I}}
\newcommand{\J}{\mathcal{J}}
\newcommand{\K}{\mathcal{K}}
\renewcommand{\L}{\mathcal{L}}
\newcommand{\M}{\mathcal{M}}
\newcommand{\N}{\mathcal{N}}
\renewcommand{\P}{\mathcal{P}}
\newcommand{\R}{\mathcal{R}}
\renewcommand{\S}{\mathcal{S}}
\renewcommand{\O}{\mathcal{O}}
\newcommand{\W}{\mathcal{W}}
\newcommand{\gr}{\mathrm{gr}}
\newcommand{\ip}{\operatorname{ip}}
\newcommand{\lgr}{\underline{\gr}}
\newcommand{\ugr}{\overline{\gr}}
\newcommand{\Grid}{\operatorname{Grid}}
\newcommand{\zpm}{0/\mathord{\pm} 1}
\newcommand{\st}{\operatorname{st}}
\newcommand{\m}{\mathbf{m}}
\renewcommand{\l}{\mathbf{l}}
\renewcommand{\r}{\mathbf{r}}
\newcommand{\f}{\mathbf{f}}

\begin{abstract}
We describe a practical algorithm which computes the accepting automaton for the insertion encoding of a permutation class, whenever this insertion encoding is regular.  This algorithm is implemented in the accompanying Maple package {\sc InsEnc}, which can automatically compute the rational generating functions for such classes.
\end{abstract}

\section{Introduction}

Permutation classes, or restricted permutations, have received considerable attention over the past two decades, and during this time a great variety of techniques have been used to enumerate them.  One of the most popular approaches, pioneered by Chung, Graham, Hoggatt, and Kleiman~\cite{chung:the-number-of-b:}, employs generating trees.  The permutation classes with finitely labeled generating trees were characterized in Vatter~\cite{vatter:finitely-labele:}.  A more powerful technique based on formal languages and called the insertion encoding was later introduced by Albert, Linton, and Ru\v{s}kuc~\cite{albert:the-insertion-e:}.  While they characterized the classes that possess regular insertion encodings, naively employing their techniques requires the determinization of non-deterministic automata several times, and no implementation has been available.  We study regular insertion encodings from a new point of view, essentially focusing on accepting automata instead of languages.  This leads both to an implementation (the Maple package {\sc InsEnc}, available for download from the author's homepage) and to a new proof of the characterization of permutation classes with regular insertion encodings.

We begin with definitions.  Two sequences of natural numbers are said to be {\it order isomorphic\/} if they have the same pairwise comparisons, so $9,1,6,7,2$ is order isomorphic to $5,1,3,4,2$.  Every sequence $w$ of natural numbers without repetition is order isomorphic to a unique permutation that we denote by $\st(w)$, so $\st(9,1,6,7,2)=5,1,3,4,2$, which we shorten to $51342$.  We call $\st(w)$ the {\it standardization\/} of $w$.  We further say that the permutation $\pi$ {\it contains\/} the permutation $\beta$ if $\pi$ contains a subsequence that is order isomorphic to $\beta$, and in this case we write $\beta\le\pi$. For example, $391867452$ contains $51342$, as can be seen by considering the subsequence $91672$.  If $\pi$ does not contain $\beta$, then $\pi$ is said to {\it avoid\/} $\beta$.

A {\it permutation class\/} is a lower order ideal in the containment ordering, meaning that if $\pi$ is contained in a permutation in the class, then $\pi$ itself lies in the class.  Permutation classes can be specified in numerous ways, but we focus solely on the most common method, in which the minimal permutations {\it not\/} in the class are given (this set is called the {\it basis\/}).  By the minimality condition, bases are necessarily {\it antichains\/}, meaning that no element of a basis is contained in another.  Although there are infinite antichains of permutations (see Atkinson, Murphy, and Ru\v{s}kuc~\cite{atkinson:partially-well-:} for constructions and references to earlier work), we restrict our attention to finitely based classes.  Given a set of permutations $B$, we define $\Av(B)$ to be the set of permutations that avoid every permutation in $B$.  Thus if $\C$ is a closed class with basis $B$ then $\C=\Av(B)$, and for this reason the elements of a permutation class are often referred to as {\it restricted permutations\/}.  We let $\C_n$ denote the set of permutations of length $n$ in $\C$ and refer to $\sum |\C_n|x^n$ as the generating function of $\C$.  All generating functions herein include the empty permutation of length $0$.

\section{Finitely Labeled Generating Trees}

In the generating tree approach to enumerating $\Av(B)$, the first step is to construct the {\it pattern-avoidance tree\/} $T(B)$ in which the children of the permutation $\pi\in\Av_{n-1}(B)$ are all permutations in $\Av_n(B)$ which can be formed by inserting $n$ into $\pi$.  The {\it active sites\/} of the permutation $\pi\in\Av_{n-1}(B)$ are defined as the indices $i$ for which inserting $n$ immediately before $\pi(i)$ produces a $B$-avoiding permutation (we also allow for $n$ to be an active site).  Thus every permutation has as many children in the pattern-avoidance tree as it has active sites.  A {\it principle subtree\/} in $T(B)$ is a subtree consisting of a single permutation and all of its descendants.

If it happens that every principle subtree in the pattern-avoidance tree $T(B)$ belongs to one of a finite number of isomorphism classes, then the (rational) generating functions for $\Av(B)$ can be easily computed using the transfer matrix method (see Flajolet and Sedgewick~\cite[Section V.6]{flajolet:analytic-combin:}).  Following \cite{vatter:finitely-labele:}, we say that the entry $\pi(i)$ of $\pi$ is {\it generating-tree-reducible\/} (relative to $B$), or simply {\it GT-reducible\/}, if the principle subtree rooted at $\pi$ and the principle subtree rooted at $\st(\pi-\pi(i))$ are isomorphic as rooted trees.  (Here $\pi-\pi(i)$ denotes the sequence obtained by deleting the entry $\pi(i)$ from $\pi$.)  We can now state the main theorem of Vatter~\cite{vatter:finitely-labele:}:

\begin{theorem}[Vatter~\cite{vatter:finitely-labele:}]\label{thm-finlabel-main}
For a finite set $B$ of permutations, the following are equivalent:
\begin{enumerate}
\item[(1)] $B$ contains both a child of an increasing permutation and a child of a decreasing permutation,
\item[(2)] there is an integer $k$ such that no node of $T(B)$ has more than $k$ children,
\item[(3)] every sufficiently long permutation in $\Av(B)$ contains a GT-reducible entry,
\item[(4)] $T(B)$ has only finitely many isomorphism classes of principle subtrees (in other words, $\Av(B)$ has a {\it finitely labeled generating tree\/}).
\end{enumerate}
\end{theorem}

The implications (3)$\implies$(4) and (4)$\implies$(2) are trivial, while (1)$\iff$(2) follows routinely from the Erd\H{o}s-Szekeres Theorem; the main content of the theorem is that (2)$\implies$(3).

\section{The Insertion Encoding}

Before describing the insertion encoding, we briefly review regular languages and finite automata.  The classic results mentioned here are covered more comprehensively in many texts, for example, Flajolet and Sedgewick~\cite[Appendix A.7]{flajolet:analytic-combin:}, so we give only the barest details.

A {\it deterministic finite automaton (DFA)\/} $M$ over the alphabet $\Sigma$ consists of a set $S$ of {\it states\/}, one of which, $s_0$, is designated the {\it initial state\/}, a {\it transition function\/} $\delta:S\times \Sigma\rightarrow S$, and a subset $A\subseteq S$ designated as {\it accept states\/}.  We denote this by $M=(S,\Sigma,s_0,A,\delta)$.  It is useful to extend the definition of the transition function $\delta$ to a map $\delta: S\times \Sigma^*\rightarrow S$ in the obvious way.  We say that the state $t$ is {\it reachable\/} from the state $s$ if there is a word $w\in\Sigma^*$ such that $\delta(s,w)=t$; otherwise, $t$ is {\it unreachable\/} from $s$.  The automaton $M$ {\it accepts\/} the word $w\in\Sigma^*$ if $\delta(s_0,w)$ is an accept state.  The set of all such words is the {\it language accepted\/} by the automaton, $\L(M)$.

A language that is accepted by a finite automaton (deterministic or not) is called {\it recognisable\/}.  By Kleene's Theorem, the recognisable languages are precisely the {\it regular languages\/}, and our purposes the reader may simply take this as the definition of regular languages.  Regular languages have numerous pleasing properties, but for our purposes we only care that they have rational generating functions which may be computed from their accepting automata.

Central to the insertion encoding is the notion of a configuration.  Consider, for example, the permutation $\pi=423615$.  In the generating tree viewpoint, $\pi$ is a descendant of $4231$.  In the insertion encoding viewpoint, we note that $\pi$ is obtained from $4231$ by inserting entries between the $3$ and the $1$ and after the $5$.  Thus we say that $\pi$ evolves from the configuration $423\diamond 1\ \diamond$.  Formally, a {\it configuration\/} is a permutation together with zero or more $\diamond$ entries called {\it slots\/}, which may not be adjacent and must eventually be filled.  Permutations correspond to the {\it slotless configurations\/}.

Given a configuration, there are four different ways to insert a new maximum entry $m$ into the $i$th slot (we number slots from left to right) of a configuration.  We may insert this maximum entry into the middle of the slot (replacing the $\diamond$ with $\diamond m\diamond$), to the left of the slot (replacing the $\diamond$ with $m\diamond$), to the right of the slot (replacing the $\diamond$ with $\diamond m$), or we may fill the slot (replacing the $\diamond$ with $m$).  These four types of operations are denoted by $\m_i$, $\l_i$, $\r_i$, and $\f_i$, respectively.  This gives a unique encoding of every permutation, called the {\it insertion encoding\/}%
\footnote{This encoding actually dates back at least to the work of Fran{\c{c}}on and Viennot~\cite{francon:permutations-se:}, although instead of words, they used colored Motzkin paths.}.
For example, the insertion encoding of $423615$ is $\m_1\m_2\l_2\f_1\f_2\f_1$.  The insertion encoding of a permutation class $\Av(B)$ is the language consisting of the insertion encodings of every element of $\Av(B)$.  We say that a configuration is {\it valid\/} for $\Av(B)$ if it can be filled to produce a permutation in $\Av(B)$.

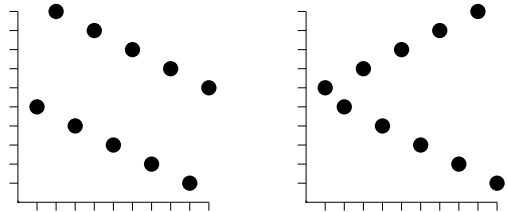
\begin{figure}
\begin{center}
\begin{tabular}{ccc}
\psset{xunit=0.01in, yunit=0.01in}
\psset{linewidth=0.005in}
\begin{pspicture}(0,0)(100,100)
\psaxes[dy=10, Dy=1, dx=10, Dx=1, tickstyle=bottom, showorigin=false, labels=none](0,0)(100,100)
\pscircle*(10,50){0.04in}
\pscircle*(20,100){0.04in}
\pscircle*(30,40){0.04in}
\pscircle*(40,90){0.04in}
\pscircle*(50,30){0.04in}
\pscircle*(60,80){0.04in}
\pscircle*(70,20){0.04in}
\pscircle*(80,70){0.04in}
\pscircle*(90,10){0.04in}
\pscircle*(100,60){0.04in}
\end{pspicture}
&\rule{10pt}{0pt}&
\psset{xunit=0.01in, yunit=0.01in}
\psset{linewidth=0.005in}
\begin{pspicture}(0,0)(100,100)
\psaxes[dy=10, Dy=1, dx=10, Dx=1, tickstyle=bottom, showorigin=false, labels=none](0,0)(100,100)
\pscircle*(10,60){0.04in}
\pscircle*(20,50){0.04in}
\pscircle*(30,70){0.04in}
\pscircle*(40,40){0.04in}
\pscircle*(50,80){0.04in}
\pscircle*(60,30){0.04in}
\pscircle*(70,90){0.04in}
\pscircle*(80,20){0.04in}
\pscircle*(90,100){0.04in}
\pscircle*(100,10){0.04in}
\end{pspicture}
\end{tabular}
\end{center}
\caption{A vertical parallel alternation (left) and a vertical wedge alternation (right).}
\label{fig-wedge-parallel}
\end{figure}

For any regular language $\L$, there is an integer $k$ such that if $w$ is a prefix of a word in $\L$, then $w$ is a prefix of a word in $\L$ with at most $k$ additional symbols%
\footnote{To see this, suppose that the accepting automaton, say $(S,\Sigma,s_0,A,\delta)$, for $\L$ has $k+1$ states.  If there is a path from $\delta(s_0,w)$ to an accept state, then there is a path from $\delta(s_0,w)$ to an accept state which does not revisit any states, and thus has length at most $k$.}
Thus for the insertion encoding of $\Av(B)$ to be regular, there must be a bound on the number of slots in valid configurations for $\Av(B)$; if $\Av(B)$ satisfies this constraint, we call it {\it slot-bounded\/}.  Thus $\Av(B)$ cannot contain arbitrarily long {\it vertical alternations\/}, which are permutations in which every even (resp., odd) indexed entry lies above every odd (resp., even) index entry.  By the Erd\H{o}s-Szekeres Theorem~\cite{erdos:a-combinatorial:}, every long vertical alternation contains a long vertical parallel alternation or a long vertical wedge alternation (see Figure~\ref{fig-wedge-parallel}), which makes it easy to check if the insertion encoding of a class needs only a finite alphabet%
\footnote{More explicitly, $\Av(B)$ contains only finitely many vertical parallel alternations oriented as in Figure~\ref{fig-wedge-parallel} if and only if $B$ contains a member of $\Av(123,3142,3412)$, while it contains only finitely many vertical wedge alternations oriented as in Figure~\ref{fig-wedge-parallel} if and only if $B$ contains a member of $\Av(132,312)$.  One must also check the reversals of these classes, $\Av(321,2143,2413)$ and $\Av(213,231)$.}.  This necessary condition is also sufficient:

\begin{theorem}[Albert, Linton, and Ru\v{s}kuc~\cite{albert:the-insertion-e:}]\label{thm-insenc-main}
For a finite set $B$ of permutations, the following are equivalent:
\begin{enumerate}
\item[(1)] $\Av(B)$ contains only finitely many vertical alternations,
\item[(2)] there is an integer $k$ such that no valid configuration for $\Av(B)$ has more than $k$ slots,
\item[(4)] the insertion encoding of $\Av(B)$ is regular.
\end{enumerate}
\end{theorem}

\section{Our Approach to the Insertion Encoding}

While Albert, Linton, and Ru\v{s}kuc considered the insertion encoding from the viewpoint of formal languages, our approach parallels that of Theorem~\ref{thm-finlabel-main}, and borrows terminology from the minimization of DFAs.

Given a DFA $M=(S,\Sigma,s_0,A,\delta)$ over the alphabet $\Sigma$ with no unreachable states, we say that $w\in\Sigma^*$ {\it distinguishes\/} between two states $s$ and $t$ of $M$ if $\delta(s,w)\in\L(M)$ and $\delta(t,w)\notin\L(M)$, or vice versa.  Two states are called {\it indistinguishable\/} if there is no word which distinguishes them.  This defines an equivalence relation on the states of $M$; for a state $s\in S$, we let $[s]$ denote the equivalence class of all states which are indistinguishable from $s$.  By the Myhill-Nerode Theorem, this equivalence relation describes the {\it minimal automation for $\L(M)$\/}, that is, the automaton with the minimum possible number of states which accepts the language $\L(M)$.

\newtheorem*{myhill-nerode-theorem}{The Myhill-Nerode Theorem~\cite{myhill:finite-automata:, nerode:linear-automato:}}
\begin{myhill-nerode-theorem}
Let $M=(S,\Sigma,s_0,A,\delta)$ be a DFA with no states which are unreachable from $s_0$.  The DFA $\tilde{M}=(\tilde{S},\Sigma,[s_0],\tilde{A},\tilde{\delta})$ where $\tilde{S}=\{[s] : s\in S\}$, $\tilde{A}=\{[s] : s\in A\}$, and $\tilde{\delta}([s],w)=[sw]$ is the minimum DFA for $\L(M)$.
\end{myhill-nerode-theorem}

Suppose that we are given the basis $B$ for a slot-bounded class $\Av(B)$, and that we would like to construct the accepting automaton for the insertion encoding of $\Av(B)$.  We could build an infinite accepting automaton in which the states of the automata are the valid configurations for $\Av(B)$.  The initial state would be $\diamond$, while the accept states would be the slotless configurations, and the transitions would be the obvious transitions given by inserting in the middle of a slot, to the left or right, or filling the slot.

In order to construct a finite automaton which accepts the insertion encoding of $\Av(B)$, we essentially minimize this infinite automaton, although we focus only on a special type of indistinguishable states.  We say that the entry $c(i)$ in the configuration $c$ is {\it insertion-encoding-reducible\/} (relative to $B$), or simply {\it IE-reducible\/} if $c$ is indistinguishable from $\st(c-c(i))$, where here we have extended the notion of standardization to states in the obvious manner, e.g., $\st(9\diamond 1672\ \diamond)=5\diamond 1342\ \diamond$.  Note, trivially, that $c(i)$ will not be IE-reducible if it is a $\diamond$ or if has $\diamond$s to both of its sides, as then $c(i)$ and $\st(c-c(i))$ will have a different number of slots.

In the next proposition, which verifies that IE-reducibility is decidable, we say that the word $w$ {\it weakly distinguishes\/} the states $s$ and $t$ if one can reach an accept state from $\delta(s,w)$ but not from $\delta(t,w)$, or vice versa.  Note that a weakly distinguishing word need not distinguish the two states, but its existence implies that the states are distinguished by some word.

\begin{proposition}\label{prop-insenc-test-IE}
Let $c$ be a valid configuration for $\Av(B)$ and suppose that the longest element of $B$ has length $b$.  If the entry $c(i)$ is neither a $\diamond$ nor adjacent to two $\diamond$s, then it is IE-reducible if and only if no word of length at most $b-1$ weakly distinguishes $c$ and $\st(c-c(i))$.
\end{proposition}
\begin{proof}
If the entry $c(i)$ is IE-reducible then $c$ and $\st(c-c(i))$ are equivalent, so they are not weakly distinguished by any words, let alone those of length at most $b-1$.

Suppose then that the entry $c(i)$ is not IE-reducible, so $c$ and $\st(c-c(i))$ are distinguished by some word $w$.  If $\delta(c,w)$ is an accept state then $\delta(\st(c-c(i)), w)$ must be as well, so we may assume that $\delta(c,w)$ is not an accept state but $\delta(\st(c-c(i)), w)$ is.  Because $c$ and $\st(c-c(i))$ have the same number of slots (since $c(i)$ is neither a $\diamond$ nor adjacent to two $\diamond$s), this means that $\delta(c,w)$ is a permutation, say $\pi\notin\Av(B)$.  Choose some copy of a basis element $\beta\in B$ in $\pi$.  Because $\st(\pi-c(i))\in\Av(B)$, this copy of $\beta$ contains $c(i)$ together with at most $b-1$ other entries.  By ignoring all entries of $\pi$ which are neither in (the underlying permutation of) $c$ nor in this copy of $\beta$, we see that there is a word $v$ of length at most $b-1$ which weakly distinguishes $c$ from $\st(c-c(i))$, as desired.
\end{proof}

We are now ready to state and prove our strengthening of Theorem~\ref{thm-insenc-main}.

\newtheorem*{thm-insenc-main-reducible}{Theorem~\ref{thm-insenc-main}'}
\begin{thm-insenc-main-reducible}
For a finite set $B$ of permutations, the following are equivalent:
\begin{enumerate}
\item[(1)] $\Av(B)$ contains only finitely many vertical alternations,
\item[(2)] there is an integer $k$ such that no valid configuration for $\Av(B)$ has more than 
$k$ slots,
\item[(3)] every sufficiently long configuration contains an IE-reducible entry,
\item[(4)] the insertion encoding of $\Av(B)$ is regular.
\end{enumerate}
\end{thm-insenc-main-reducible}

As with Theorem~\ref{thm-finlabel-main}, note that three of the implications in Theorem~\ref{thm-insenc-main}' are trivial.  We have already remarked that (4)$\implies$(2) and (2)$\iff$(1), while (3)$\implies$(4) because (3) shows that the insertion encoding for $\Av(B)$ has a finite accepting automaton.  Only (2)$\implies$(3) remains.

\newenvironment{thm-insenc-main-reducible-proof}{\medskip\noindent {\it Proof that (2)$\implies$(3) in Theorem~\ref{thm-insenc-main}'.\/}}{\qed\bigskip}
\begin{thm-insenc-main-reducible-proof}
Suppose that the longest element of $B$ has length $b$.  We are given that no valid configuration (for $\Av(B)$) has more than $k$ slots for some $k$, and must show that every sufficiently long configuration contains an IE-reducible entry.

Given a configuration $c$ of length $n$, let $I\subseteq[n]$ denote the set of indices such that $c(i)$ is neither a $\diamond$ nor adjacent to a $\diamond$.  Note that since no valid configuration has more than $k$ slots, all but a bounded number of indices lie in $I$.

If $c(i)$ is not IE-reducible for some $i\in I$, then the proof of Proposition~\ref{prop-insenc-test-IE} shows that there is a word $u$ of length at most $b-1$ which leads to a valid configuration from $\st(c-c(i))$ but not from $c$.  Since $\delta(\st(c-c(i)), u)$ is a valid configuration with at most $k$ states, there must be a word $v$ of length at most $k$ which leads from this configuration to a valid slotless configuration, i.e., a permutation in $\Av(B)$.  Therefore, the word $w=uv$ leads from $\st(c-c(i))$ to a permutation in $\Av(B)$, while it leads to an invalid configuration from $c$.  We call $w$ a {\it witness\/} for $i$.  Every occurrence of an element of $B$ in $\delta(c,w)$ must include the entry $c(i)$, and so no word may witness more than $b$ elements of $I$.  Therefore, since there are a bounded number of possible witnesses (they can each have length at most $b+k-1$), each can witness at most $b$ indices of $I$, and $I$ contains all but a bounded number of indices, for $n$ sufficiently large, there must be at least one index of $I$ without a witness.  Clearly this witness-less element of $I$ is IE-reducible, completing the proof.
\end{thm-insenc-main-reducible-proof}

It follows easily from the definitions that if $\pi(i)$ is GT-reducible (relative to $B$) for the permutation $\pi$, then in every configuration whose underlying permutation is $\pi$, the entry corresponding to $\pi(i)$ is IE-reducible.  This verifies that every class with a finitely labeled generating tree also has a regular insertion encoding.

\section{Counting Sum Indecomposable Permutations}

The permutation $\pi$ of length $n$ is {\it sum indecomposable\/} (or, {\it connected\/}) if there is no integer $2\le i\le n-1$ such that $\pi(\{1,2,\dots,i\})=\{1,2,\dots,i\}$.

As we describe first, it is fairly easy to characterize the sum indecomposable permutations via the insertion encoding.  The evolution of a sum {\it decomposable\/} permutation must contain a non-initial configuration whose only slot occurs at the end of the configuration.  Conversely, every permutation which can be formed from such a configuration is sum decomposable.  If $\Av(B)$ has a regular insertion encoding, then we know by Theorem~\ref{thm-insenc-main} that there is a constant $k$ such that no valid configuration for $\Av(B)$ has more than $k$ slots.  In order to recognize the sum {\it indecomposable\/} permutations in this class, we therefore need only to keep track of how many open slots a configuration has and whether the rightmost slot occurs at the end of the configuration, rejecting a permutation whenever its evolution includes a non-initial configuration whose only slot occurs at the end of the configuration.  It follows from the closure properties of regular languages that the sum indecomposable permutations in a class with a regular insertion encoding also have a regular insertion encoding.

While this shows that the sum indecomposable permutations in a class $\Av(B)$ with a regular insertion encoding themselves have a regular insertion encoding, it describes a rather circuitous route to this encoding.  Instead, a straight-forward adaptation of our approach leads directly to the accepting automaton for the insertion encoding of sum indecomposable permutations in $\Av(B)$.  Let us say that the element $c(i)$ of the configuration $c$ is {\it SIE-reducible\/} (relative to $B$) if it is IE-reducible and is not the rightmost entry of $c$.  It follows from our proof of Theorem~\ref{thm-insenc-main}' that every sufficiently long valid configuration for $\Av(B)$ has more than one IE-reducible entry, and so has at least one SIE-reducible configuration.  To construct the accepting automaton for the sum indecomposable permutations in $\Av(B)$, one therefore eliminates all configurations whose only slot occurs at the end, and identifies $c$ and $\st(c-c(i))$ whenever $c(i)$ is SIE-reducible.  This is also implemented in the package {\sc InsEnc}.

\section{Conclusion}

We have presented a new viewpoint of regular insertion encodings, which has lead to a new proof of Theorem~\ref{thm-insenc-main} and to an implementation in the Maple package {\sc InsEnc}, available from the author's homepage.  We conclude with some results obtained from this package.

Recall that the three permutation class symmetries inverse ($\pi\mapsto\pi^{-1}$), reverse $(\pi\mapsto\pi(n)\cdots\pi(2)\pi(1)$), and complement ($\pi\mapsto(n+1-\pi(1))(n+1-\pi(2))\cdots(n+1-\pi(n))$) generate the symmetries of the square.  Given a set of permutation classes, it is therefore useful to divide them into symmetry classes.  For example, the ${4!\choose 2}=276$ permutation classes with precisely two basis elements of length $4$ fall into $56$ distinct symmetry classes.  Two permutation classes are further said to be {\it Wilf-equivalent\/} if they are equinumerous.  Le~\cite{le:wilf-classes-of:} recently established that these $56$ symmetry classes form $38$ distinct Wilf classes.  Of these $38$, $12$ can be enumerated with regular insertion encodings.  Of those $12$, $10$ of those can be enumerated using finitely labeled generating trees, and their generating functions are reported in Vatter~\cite{vatter:finitely-labele:} (these generating functions were also computed by hand by Kremer and Shiu~\cite{kremer:finite-transiti:}).  The $2$ new generating functions are listed below.
$$
\begin{array}{l|l}
\mbox{Class}				&	\mbox{Generating function}\\[2.5pt]
\hline\hline\\[-8pt]
\Av(4321,1324)		&
	\frac{1-11x+56x^2-172x^3+357x^4-519x^5+554x^6-413x^7+217x^8-83x^9+20x^{10}-2x^{11}}{(1-x)^{12}}\\[2.5pt]
\Av(4321,3142)		&
	\frac{(1-x)(1-3x)^2}{(1-2x)^2(1-4x+x^2)}\\
\end{array}
$$

From the generating function displayed above, it follows that for large $n$, the number of permutations in $\Av(4321,1324)$ of length $n$ is given by a polynomial.  This is not a surprise, as it can be checked that this class meets the conditions of Huczynska and Vatter~\cite{huczynska:grid-classes-an:} or Albert, Atkinson, and Brignall~\cite{albert:permutation-cla:} who, building on the work of Kaiser and Klazar~\cite{kaiser:on-growth-rates:}, characterized the permutation classes of polynomial growth.  In fact, it follows from Theorem~\ref{thm-insenc-main} that all permutation classes with lower growth rate%
\footnote{The {\it lower growth rate\/} of the permutation class $\C$ is $\displaystyle\liminf_{n\rightarrow\infty}\sqrt[n]{|\C_n|}$.}
less than $2$ have regular insertion encodings.

\bigskip
\noindent{\bf Acknowledgments:} The author thanks Michael Albert for numerous helpful suggestions, and for pointing out the work of Fran{\c{c}}on and Viennot~\cite{francon:permutations-se:}.

\bibliographystyle{acm}
\bibliography{../refs}

\end{document}